\documentclass[reqno]{amsart}
\usepackage[utf8]{inputenc}
\usepackage[margin = 1in]{geometry}
\usepackage{latexsym,amsmath,color,amsthm,epsfig,mathrsfs,mathdots,enumerate}
\usepackage{mathtools}
\mathtoolsset{showonlyrefs}
\usepackage{graphicx}
\usepackage{amssymb}
\usepackage{mdframed}
\usepackage{fancyhdr}
\usepackage{tikz-cd}
\usepackage{verbatim}
\usepackage{nicefrac}
\usepackage{bbm}
\usepackage{hyperref}
\hypersetup{
    colorlinks = true,
    citecolor=black,
    filecolor=black,
    linkcolor=black,
    urlcolor=blue
}
\usepackage{amsfonts, esint}
\usepackage{color}
\usepackage{dsfont}

\numberwithin{equation}{section}

 \usepackage{lastpage}
\newtheorem{thm}{Theorem}[section]
\newtheorem{prop}[thm]{Proposition}
\newtheorem{cor}[thm]{Corollary}

\newtheorem{lemma}[thm]{Lemma}
\newtheorem{defn}[thm]{Definition}

\newtheorem{preremark}[thm]{Remark}

\newenvironment{remark}{\begin{preremark}\rm}{\medskip \end{preremark}}
\numberwithin{equation}{section}

\newcommand{\norm}[1]{\left\Vert#1\right\Vert}

\newcommand{\R}{\mathbb R}

\newcommand{\grad} {\nabla}
\newcommand{\lap} {\Delta}

\newcommand{\dd} {\; \mathrm{d}}

\newcommand{\RP}{ \mathbb R \text{P}}

\DeclareMathOperator{\tr}{tr}

\definecolor{sh}{RGB}{255,0,100}

\begin{document}

\begin{abstract}
    We prove upper and lower bounds on the optimal constant $\Lambda_d$ of the Bakry-\'Emery $\Gamma_2$ criterion for positive symmetric functions on the unit sphere $S^{d-1}$, which also can be identified as positive functions on the real projective space $\RP^{d-1}$. The Bakry-\'Emery $\Gamma_2$ criterion inequality was crucially used to prove the monotonicty of the Fisher information for the Landau equation by Guillen and Silvestre \cite{guillen2023global} recently. Therefore, a better bound on the optimal constant $\Lambda_d$ expands the range of interaction potentials that exhibits the monotonicity of the Fisher information.
    In particular, we compute that $\Lambda_3$ is between $5.5$ and $5.739$.
\end{abstract}

\title{Bounds for the optimal constant of the Bakry-\'Emery $\Gamma_2$ criterion inequality on $\RP^{d-1}$ }
\author{Sehyun Ji}
\address[Sehyun Ji]{Department of Mathematics, University of Chicago,  Chicago, Illinois 60637, USA}
\email{jise0624@uchicago.edu}
\maketitle

\section{Introduction}
 In this work, we study the largest constant $\Lambda_d$ such that the inequality
\begin{equation} \label{eq: 2nd logSobolev}
    \int_{S^{d-1}} f\Gamma_2(\log f, \log f) \dd \sigma \ge \Lambda_d \int_{S^{d-1}} f |\grad_\sigma \log f|^2 \dd \sigma
\end{equation}
holds for any positive symmetric function $f$ on the unit sphere $S^{d-1} \subset \R^d$
where
\begin{equation}
    \Gamma_2(g,g)=  \norm{\grad_\sigma^2 g}^2 + (d-2)|\grad_\sigma g|^2.
\end{equation}

Here  $\grad_\sigma$ is the spherical gradient of $S^{d-1}$ at $\sigma \in S^{d-1}$ and $\dd \sigma$ is a normalized surface measure of $S^{d-1}$: 
\begin{equation}
    \int_{S^{d-1}} \dd \sigma =1.
\end{equation}Note that the class of positive symmetric functions on the unit sphere $S^{d-1}$ can be identified as the class of positive functions on the real projective space $\RP^{d-1}$.

In the literature, the inequality \eqref{eq: 2nd logSobolev} is known as the Bakry-\'Emery $\Gamma_2$ criterion for the logarithmic Sobolev inequality. 
It has got an attention as a tool to study the logarithmic Sobolev inequalities on Riemannian manifolds, but not much as of itself.
The criterion says that if the inequality \eqref{eq: 2nd logSobolev} holds, then the logarithmic Sobolev inequality 
\begin{equation} \label{eq: log Sobolev}
     \int_{S^{d-1}} f|\grad_\sigma \log f|^2 \dd \sigma \ge 2 c \left[ \int_{S^{d-1}} f \log f \dd \sigma - \left(\int_{S^{d-1}} f \dd \sigma\right) \log \left(\int_{S^{d-1}} f \dd \sigma\right) \right]
\end{equation}
with $c=\Lambda_d$ holds. Let's denote $\alpha_d$ as the largest constant $c$ such that
the inequality \eqref{eq: log Sobolev} holds for every positive symmetric function $f$ on $S^{d-1}$.
 Then, the Bakry-\'Emery $\Gamma_2 $ criterion is implying
\begin{equation} \label{ordering1}
    \alpha_d \ge \Lambda_d.
\end{equation}
By substituting $g=\sqrt{f}$, the inequality \eqref{eq: log Sobolev} is equivalent to
\begin{equation} \label{eq: log Sobolev2}
    \int_{S^{d-1}} |\grad_\sigma g|^2 \dd \sigma \ge \frac{c}{2}  \left[\int_{S^{d-1}} g^2 \log g^2 \dd \sigma- \left( \int_{S^{d-1}}  g^2 \dd \sigma \right) \log \left( \int_{S^{d-1}} g^2 \dd \sigma \right)   \right]
\end{equation}
for a positive symmetric function $g$ on $S^{d-1}$.
Furthermore, there is a classical result that the inequality \eqref{eq: log Sobolev2} implies the Poincar\'e inequality 
\begin{equation} \label{Poin}
    \int_{S^{d-1}} |\grad_\sigma g|^2 \dd \sigma \ge c \int_{S^{d-1}} \left(g-\int_{S^{d-1}}g \dd \sigma \right)^2 \dd \sigma= c \left(\int_{S^{d-1}} g^2 \dd \sigma -\left(\int_{S^{d-1}} g \dd\sigma\right)^2 \right).
\end{equation}
Let's denote $\lambda_d$ as the largest constant $c$ such that the inequality \eqref{Poin} holds for every positive symmetric function $g$ on $S^{d-1}$, then
\begin{equation}
    \lambda_d \ge \alpha_d.
\end{equation}
Moreover, $\lambda_d$ is the first nonzero eigenvalue of the Laplace-Beltrami operator $-\lap$. For the class of positive symmetric functions on $S^{d-1}$, it is elementary that $\lambda_d=2d$ by spherical harmonics (See Lemma \ref{lem: Poincare} for the explanation).
From the above discussion, we have the ordering
\begin{equation} \label{ineq: ordering}
    \lambda_d =2d \ge \alpha_d \ge \Lambda_d.
\end{equation}

In the present paper, we focus on the inequality \eqref{eq: 2nd logSobolev} itself and prove bounds for $\Lambda_d$. This work is inspired by the recent breakthrough of Guillen and Silvestre \cite{guillen2023global} of proving the existence of global smooth solutions for the Landau equation with Coulomb potentials.
The inequality \eqref{eq: 2nd logSobolev} is crucially used in the proof and they call it as a Poincar\'e-like inequality . In particular, they proved $\Lambda_3 \ge \frac{19}{4}$. It is worth to mention that $\Lambda_3\ge \frac{9}{4}$ is needed to cover the Coulomb potential case.\\

Our first main result is on the lower bound for $\Lambda_d$.
\begin{thm}[The lower bound for $\Lambda_d$]\label{thm: Thm 1}
    Let $f:S^{d-1}\to \R_{>0}$ be a function defined on $S^{d-1}$ and symmetric, i.e. $f(\sigma)=f(-\sigma)$ for every $\sigma \in S^{d-1}$. The optimal constant $\Lambda_d$ for the inequality \eqref{eq: 2nd logSobolev} has the following lower bound:
    \begin{equation}
        \Lambda_d \ge d+3-\frac{1}{d-1}.
    \end{equation}
    In particular, $\Lambda_2 \ge 4$ and $\Lambda_3 \ge 5.5$.
\end{thm}
\begin{remark}
    It was already observed in \cite[Lemma A.2.] {guillen2023global} that $\Lambda_2=4$. We can deduce the same conclusion from Theorem \ref{thm: Thm 1} and the ordering \eqref{ineq: ordering}.
\end{remark}
As a corollary, we can improve the theorem regarding the monotonicity of the Fisher information by Guillen and Silvestre.
\begin{cor} [Improvement of Theorem 1.1 of \cite{guillen2023global}]
    Let $f:[0,T]\times \R^3$ be a classical solution for the space-homogeneous Landau equation of the following form:
    \begin{equation}
        \partial_t f =\partial_{v_i} \int_{\R^3} \alpha(|v-w|)a_{ij}(v-w) (\partial_{v_j}-\partial_{w_j}) [f(v)f(w)] \dd w
    \end{equation}
    for $a_{ij}(z)=|z|^2\delta_{ij}-z_iz_j$ with the interaction potential $\alpha$. If $\alpha$ satisfies
    \begin{equation}
        \sup_{r>0} \frac{r|\alpha'(r)|}{\alpha(r)} \le \sqrt{22},
    \end{equation}
    then the Fisher information $i(f)$ is decreasing in time.
    
\end{cor}

\begin{remark}
    Without the symmetry assumption, it is well known that
    \begin{equation}
        \lambda_d=\alpha_d=\Lambda_d=d-1,
    \end{equation}
    and $\Lambda_3=2$ in particular. Thus, the symmetry assumption is needed to prove the existence of global smooth solutions for the Landau equation with Coulomb potentials. 
\end{remark}

Our second main result is improving the obvious upper bound $6$ for $\Lambda_3$. We only consider the case $d=3$ since it comes from the original motivation of the Landau equation.
 Our method for improving the upper bound for $\Lambda_3$ is plugging in an explicit class of functions to the inequality \eqref{eq: 2nd logSobolev}. 
\begin{thm}[The upper bound for $\Lambda_3$]\label{thm: upper}
    Let $h:S^2 \to \R$ be 
    \[
     h(x,y,z)=(z^2+t)^2
    \]
    for a parameter $t\in \R_{> 0}$. Then,
    \begin{equation}
        \Lambda_3 \le \min_{t>0} \left[ 5 (3t+1)(3t+2) - 15(t+1)(3t+1)\sqrt{t}\arctan\left(\frac{1}{\sqrt{t}}\right) \right] \approx 5.73892.
    \end{equation}
\end{thm}
The minimum $5.73892\cdots$ is achieved at $t\approx 0.69214$ and computed numerically. It is worth noting that plugging in the same class of functions from Theorem \ref{thm: upper} to the logarithmic Sobolev inequality \eqref{eq: log Sobolev} provides the upper bound for $\alpha_3$:
\begin{thm}[The upper bound for $\alpha_3$] \label{thm: upperalpha}
    \begin{equation}
        \alpha_3 \le \min_{t>0}\left[ \left(  2 t^2 \sqrt{t}\arctan\left(\frac{1}{\sqrt{t}}\right) +\frac{15}{16}\left(t^2+\frac{2t}{3}+\frac15 \right)\log\left(\frac{t^2+2t+1}{t^2+\frac{2t}{3}+\frac15}\right)-\frac{120t^2+35t+9}{60} \right)^{-1} \right] \approx 5.8358.
    \end{equation}
    
\end{thm}
The minimum $5.8358\cdots$ is achieved at $t \approx 0.757585$ and computed numerically.
Finding the precise values of $\alpha_3$ and $\Lambda_3$ still remains as an interesting question to be answered.

\begin{remark}
    It is not difficult to verify that Remark 5.1.5 of \cite{bakry2014markov} implies $\alpha_3 <\lambda_3=6$. In fact, $\alpha_d<\lambda_d$ holds for every $d\ge 3$.
    Theorem \ref{thm: upperalpha} provides a quantitative estimate of the difference between $\alpha_3$ and $\lambda_3$.
    However, it is unclear whether $\alpha_3=\Lambda_3$ or $\alpha_3>\Lambda_3$. 
\end{remark}

\section{Preliminaries}
\subsection{The logarithmic Sobolev inequality and the Curvature-Dimension condition}
The logarithmic Sobolev inequality has been studied extensively following the seminal paper by Gross \cite{gross1975log}. Let's start by recalling the definitions of the Carr\'e du champ operator and the iterated Carr\'e du champ opertator in the context of diffusion Markov generators $L$ on some measure space $M$. In this paper, we focus on the case where $M$ is a Riemannian manifolds.

\begin{defn}[Carr\'e du champ operator]
    The Carr\'e du champ operator $\Gamma$ is defined as
    \begin{equation}
        \Gamma(f,g):=\frac12 \left(L(fg)-f Lg-gLf \right).
    \end{equation}
    Moreover, the iterated Carr\'e du champ operator $\Gamma_2$ is defined as 
    \begin{equation}
        \Gamma_2(f,g):=\frac12 \left(L\Gamma(f,g)-\Gamma(f,Lg)-\Gamma(Lf,g) \right).
    \end{equation}

\end{defn}
  In particular, for a Riemannian manifold $M$ and the Laplace-Beltrami operator $L =\lap$, 
    \begin{equation}
        \Gamma(f,g)=\grad f \cdot \grad g,
    \end{equation}
    \begin{equation}
        \Gamma_2(g,g)=\frac12 \lap (|\grad g|^2)-\grad g \cdot \grad \lap g= \norm{\grad^2 g}^2+Ric(\grad g,\grad g).
    \end{equation}
 Note that the second equation follows from the Bochner's formula. With the new notations, the inequality \eqref{eq: 2nd logSobolev} can be written as
 \begin{equation} \label{eq: second log Sobolev form}
    \int_{S^{d-1}} f (\norm{\grad_\sigma^2 \log f}^2+(d-2) |\grad_\sigma \log f|^2) \dd \sigma  \ge \Lambda_d \int_{S^{d-1}} f |\grad_\sigma \log f|^2  \dd \sigma
\end{equation}
using that the Ricci curvature of $(S^{d-1},g_{round})$ is $(d-2)g_{round}$. We also define the Curvauture-Dimension condition, which was introduced by Bakry and \'Emery in \cite{bakry1985emery} in the context of hypercontractivity of Markov diffusion operators.

\begin{defn}[The Curvature-Dimension condition]
    We say an operator $L$ on a space $M$ satisfies the Curvature-Dimension condition $CD(\rho,n)$ if 
    \begin{equation} \label{BakryEmery CD}
        \Gamma_2(g,g) \ge \frac{1}{n}(L g)^2 +\rho\Gamma(g,g)
    \end{equation}
    holds for every function $g$ on $M$.
\end{defn}
It is well known if $M$ is a Riemannian manifold of dimension $n$ with Ricci curvature bounded from below by $\rho>0$, then the Laplace-Beltrami operator $L=\lap$ satisfies the Curvature-Dimension condition $CD(\rho,n)$. Therefore, the Laplace-Beltrami operators $\lap$ of $S^{d-1}$ and $\RP^{d-1}$ satisfy the Curvature-Dimension condition $CD(d-2,d-1)$. 
We can also interpret this condition by verifying
\begin{equation} \label{eq: cd condtion}
    \norm{\grad^2_\sigma g}^2 \ge \frac{1}{d-1} (\lap_\sigma g)^2
\end{equation}
using spherical coordinates.
Indeed, the matrix $\grad_\sigma^2 g$ on $\R^d$ does not have full rank, i.e., its rank is at most $d-1$. Note that both spaces $\RP^{d-1}$ and $S^{d-1}$ have the same Ricci curvature and their Laplace-Beltrami operators satisfy the same Curvature-Dimension condition $CD(d-2,d-1)$. However, the first nonzero eigenvalues $\lambda$ of $-\lap$ are different for  $S^{d-1}$ and $\RP^{d-1}$.

A simple consequence of the Curvature-Dimension condition is the famous Lichnerowicz lower bound on the first nonzero eigenvalue $\lambda_M$ for $-\lap$ \cite{lichnerowicz1958geometry}:
\begin{equation} \label{eq: Lich}
    \lambda_M \ge \frac{\rho n}{n-1}.
\end{equation} 
Indeed, if the inequality \eqref{BakryEmery CD} for $L=\lap$ is integrated over $M$, then it gives
\begin{equation} \label{Bakry-Emery}
    \left(1-\frac1n\right)\int_M (\lap g)^2 \ge  \rho  \int_M |\grad g|^2 
\end{equation}
which implies the inequality \eqref{eq: Lich}. Another consequence of the Curvature-Dimension condition is the important result by Bakry and \'Emery \cite{bakry1985emery}
\begin{equation}  \label{BE}
    \alpha_M \ge \frac{\rho n}{n-1},
\end{equation}
 which implies the Lichnerowicz bound \eqref{eq: Lich}. Bakry and \'Emery actually prove
 \begin{equation} \label{BE2}
     \Lambda_M \ge \frac{\rho n}{n-1}
 \end{equation}
 to obtain the lower bound \eqref{BE}.
 Later, Rothaus \cite{rothaus1986hypercontractivity} improved the lower bound for $\alpha_M$ to
 \begin{equation} \label{eq: Rothaus}
     \alpha_M \ge \frac{4n}{(n+1)^2} \lambda_M +\frac{(n-1)^2}{(n+1)^2}\cdot  \frac{\rho n}{n-1} .
 \end{equation}
    
 For $M=\RP^{d-1}$, it is easy to check that the inequalities \eqref{BE2} and \eqref{eq: Rothaus} give the lower bounds for $\Lambda_d$ and $\alpha_d$ by plugging in $(\lambda_d,\rho,n)=(2d,d-2,d-1)$:
\begin{equation} \label{eq: asymp}
    \Lambda_d \ge d-1,
\end{equation}
\begin{equation} \label{eq: optimal logSobolev}
    \alpha_d \ge d+3-\frac{4}{d^2}.
\end{equation}
In particular, $\Lambda_3\ge 2$ and $\alpha_3 \ge \frac{50}{9}=5.5555\cdots$. 

We also prove that in the framework of the Curvature-Dimension condition, Theorem \ref{thm: Thm 1} can be generalized as the following.
\begin{thm}[The lower bound for $\Lambda_M$ under $CD(\rho,n)$] \label{thm: cd}
Assume the Laplace-Beltrami operator $L=\lap$ of a Riemannian manifold $M$ satisfies the Curvature-Dimension condition $CD(\rho,n)$. Then,
\begin{equation}
    \Lambda_M \ge  \frac{4n-1}{n(n+2)}  \lambda_M+ \frac{(n-1)^2}{n(n+2)} \cdot \frac{\rho n}{n-1} 
\end{equation}
\end{thm}
In particular, we recover Theorem \ref{thm: Thm 1} for $M=\RP^{d-1}$ by plugging in $(\lambda_d,\rho,n)=(2d,d-2,d-1)$.
It is obvious that Theorem \ref{thm: cd} implies the inequality \eqref{BE2}.
See \cite{bakry1994book2, bakry2014markov, bakry1994book1, fontenas1997constant, ledoux2000geometry, rothaus1981diffusion} for further discussion and examples on the logartihmic Sobolev inequality and the Curvature-Dimension condition. The stability of the logarithmic Sobolev inequality on $S^{d-1}$ is studied in the recent paper \cite{brigati2024logarithmic}, where they also consider the symmetric case as in this work.

\begin{remark}[The asymptotic behavior of $\Lambda_d$ ]
    
It is natural to ask the asymptotic behavior of $\Lambda_d$. This question can be precisely answered by the ordering \eqref{ineq: ordering} and the following theorem by Saloff-Coste \cite[Theorem 1.3]{saloff-coste1994precise}.
\[
\lim_{d\to \infty}\frac{\alpha_d}{d}=\lim_{d\to \infty} \frac{2\alpha_d}{\lambda_d}=1.
\]
Hence, the ordering \eqref{ineq: ordering} implies that
\begin{equation}
    \limsup_{d\to \infty} \frac{\Lambda_d}{d} \le 1.
\end{equation} 
Combined with the inequality \eqref{eq: asymp}, the precise asymptotic behavior of $\Lambda_d$ immediately follows. That is,

    \begin{equation}
        \lim_{d\to \infty} \frac{\Lambda_d}{d}=1.
    \end{equation}
\end{remark}

\subsection{Interpretation using the heat flow}
Now, we present the point of view that relates the inequalities \eqref{eq: 2nd logSobolev} and \eqref{eq: log Sobolev} via the heat flow of $f$ on $S^{d-1}$.
We can assume that $f$ is a probability density function on $S^{d-1}$, i.e.,
\begin{equation}
    \int_{S^{d-1}} f \dd \sigma =1,
\end{equation}
since both inequalities \eqref{eq: 2nd logSobolev} and \eqref{eq: log Sobolev} are invariant under scaling. We define the entropy functional and the Fisher information for the probability density $f$.
\begin{defn}[The entropy functional]
    \begin{equation}
        h(f):=\int_{S^{d-1}} f \log f \dd \sigma.
    \end{equation}
\end{defn}
\begin{defn}[The Fisher information]
\begin{equation} \label{eq: def}
    i(f):= \int_{S^{d-1}} f |\grad_\sigma \log f|^2 \dd \sigma = \int_{S^{d-1}} \frac{|\grad_\sigma f|^2}{f} \dd \sigma = 4 \int_{S^{d-1}} |\grad_\sigma \sqrt{f}|^2 \dd \sigma.
\end{equation}
\end{defn}
We remark that a positive symmetric function $f$ stays positive and symmetric along the heat flow $\partial_t f =\lap_\sigma f$. The time derivative of the entropy is the negative of the Fisher information. That is,
\begin{equation}
    \partial_t h(f) =-i(f).
\end{equation}
Hence, the inequality \eqref{eq: log Sobolev} provides how to bound the entropy from above by its time derivative, the Fisher information. Analogously, the time derivative of the Fisher information along the heat flow is negative two times the left hand side of \eqref{eq: 2nd logSobolev}. That is,
\begin{equation}
    \partial_t i(f)=-2 \int_{S^{d-1}} f( \norm{\grad_\sigma^2 \log f}^2 + (d-2)|\grad_\sigma \log f|^2) \dd \sigma.
\end{equation}
Therefore, the inequality \eqref{eq: 2nd logSobolev} describes how to bound the Fisher information from above by its time derivative. This is why we do not absorb the second term of the left hand side of \eqref{eq: second log Sobolev form} to the right hand side. We can rewrite the inequality \eqref{eq: 2nd logSobolev} as
\begin{equation} \label{ineq: heat}
    -\frac12 \partial_t i(f) \ge -\Lambda_d \cdot\partial_t h(f).
\end{equation}\\
This interpretation gives a short proof of the ordering 
\eqref{ordering1}. Consider the heat equation on $S^{d-1}$ with the initial data $f$. That is,
\begin{equation}
    \begin{cases}
        \partial_t f =\lap_\sigma f,\\
        f(0)=f.
    \end{cases}
\end{equation}
As time $t$ goes to $+\infty$, the function $f(t)$ converges to the constant function $1$. Subsequently, the quantities $h(f(t))$ and $i(f(t))$ converge to $0$ as $t\to \infty$.  Integrating the inequality \eqref{ineq: heat} from $t=0$ to $+\infty$, we gain
\begin{equation}
    \frac12 \left( i(f(0))-i(f(\infty))\right) \ge \Lambda_d(h(f(0))-h(f(\infty))),
\end{equation}
which is equivalent to the logarithmic Sobolev inequality \eqref{eq: log Sobolev} with $c=\Lambda_d$.
\begin{remark}
    The ordering \eqref{ordering1} can be shown in a different way. Consider a variational formulation for \eqref{eq: log Sobolev2}. It is known that for every $\tilde{\alpha}>\alpha_d$ that there exists a nonconstant, strictly positive solution $\tilde{g}$ for
    \begin{equation} \label{eq: variation}
       \frac{ \lap_\sigma \tilde{g}}{\tilde{g}} + \tilde{\alpha}  \log \tilde{g}=0.
    \end{equation}
    Denote $\tilde{f}:=\tilde{g}^2$, then \eqref{eq: variation} is equivalent to
    \begin{equation} \label{eq: variation2}
         \lap_\sigma \log (\tilde{f}) +\frac12 |\grad_\sigma \log \tilde{f}|^2 + \tilde{\alpha}  \log \tilde{f}=0.
    \end{equation}
    Multiply $\lap_\sigma \tilde{f}$ on both sides of \eqref{eq: variation2} and integrate over $S^{d-1}$, then it is easy to check 
    \begin{equation}
         \int_{S^{d-1}} \tilde{f} \Gamma_2(\log \tilde{f},\log \tilde{f} )\dd \sigma = \tilde{\alpha} i(\tilde{f}).
    \end{equation}
    It implies $\tilde{\alpha} \ge \Lambda_d$, therefore $\alpha_d \ge \Lambda_d$.

\end{remark}
\begin{remark}
    There is a known example such that $\alpha_M >\Lambda_M$. In other words, the Bakry-\'Emery $\Gamma_2$ criterion is strictly stronger than the logarithmic Sobolev inequality. See \cite{cecile2000sur, helffer2002example} for the example.
\end{remark}



\section{The lower bound for $\Lambda_d$}
Our strategy of proving the lower bound is inspired by the proofs of the inequality \eqref{BE2} by Bakry and \'Emery \cite{bakry1985emery} and the inequality \eqref{eq: Rothaus} by Rothaus \cite{rothaus1986hypercontractivity}.
Recall the Bochner's formula:
\begin{equation}
    \norm{\grad_\sigma^2 g}^2+Ric(\grad_\sigma g,\grad_\sigma g)=\frac{1}{2} \lap_\sigma (|\grad_\sigma g|^2)-\grad_\sigma \lap_\sigma g \cdot \grad_\sigma g
\end{equation}
 Since the Ricci curvature of $(S^{d-1},g_{round})$ is $(d-2)g_{round}$, we deduce that the left hand side of \eqref{eq: 2nd logSobolev} is equal to
\begin{equation} \label{eq: Bochner}
    \int_{S^{d-1}} f\left( \frac{1}{2} \lap_\sigma (|\grad_\sigma \log f|^2)-\grad_\sigma \lap_\sigma \log f \cdot \grad_\sigma \log f\right) \dd \sigma=\int_{S^{d-1}}\left( \frac{1}{2} \lap_\sigma f |\grad_\sigma \log f|^2 + \lap_\sigma f \lap_\sigma \log f \right) \dd \sigma.
\end{equation}
Note that the equality holds from integration by parts. For our convenience, we denote
\[
A:=\lap_\sigma \log f,
\]
\[
B:=|\grad_\sigma \log f|^2.
\]
Then,
\[
\frac{\lap_\sigma f}{f}= \lap_\sigma \log f+ |\grad_\sigma \log f|^2= A+B.
\]
With the new notations, the equation \eqref{eq: Bochner} is equal to
\begin{equation} \label{eq: our goal}
    \int_{S^{d-1}} f \left( A+B \right) \left(A+\frac12 B \right) \dd \sigma.
\end{equation}

On the other hand, we have the following formula:
\begin{align} \label{eq: H2norm}
    \int_{S^{d-1}} (\lap_\sigma \sqrt{f})^2 \dd \sigma 
    &=\frac14 \int_{S^{d-1}} \left( \grad_\sigma \cdot \left(\frac{\grad_\sigma f}{\sqrt{f}} \right)\right)^2 \dd \sigma\\
    &=\frac14 \int_{S^{d-1}}  f\left( \frac{\lap_\sigma f}{f}-\frac12 |\grad_\sigma \log f|^2 \right)^2 \dd \sigma\\
    &=\frac14 \int_{S^{d-1}}f \left(A+\frac12 B\right)^2 \dd \sigma.
\end{align}

\begin{lemma}[The Poincar\'e inequality on $S^{d-1}$] \label{lem: Poincare}
    Let $g:S^{d-1}\to \R$ be a symmetric function with average zero, then,
    \begin{equation}\label{eq: Poincare}
        \int_{S^{d-1}} |\grad_\sigma g|^2 \dd \sigma \ge 2d \int_{S^{d-1}}  g^2 \dd \sigma.
    \end{equation}
\end{lemma}
\begin{proof}
    This is a fact regarding the eigenvalues of $-\lap_\sigma$ of $S^{d-1}$. We briefly explain the reasoning. The eigenfunctions of $-\lap_\sigma$ correspond to spherical harmonics. Let's assume that $P$ is a spherical harmonic polynomial of degree $k$. Thus, $P(r\sigma)=r^k P(\sigma)$ for $\sigma \in S^{d-1}$. Then, $P$ is a harmonic polynomial on $\R^d$, so
    \begin{equation}
    \lap_{\R^d} P=\partial_{rr} P+\frac{d-1}{r}\partial_r P +\frac{1}{r^2} \lap_\sigma P=0.
\end{equation}
It follows that
\begin{equation}
    0=k(k-1)r^{k-2} P +(d-1)k r^{k-2} P+r^{k-2} \lap_\sigma P \Rightarrow -\lap_\sigma P= k(d+k-2)P.
\end{equation}
Hence, the eigenvalues of $-\lap_\sigma$ for $S^{d-1}$ are in form of $k(d+k-2)$ for $k=0,1,2,\cdots$. The first eigenvalue $0$ corresponds to constant functions and the second eigenvalue $d-1$ corresponds to spherical harmonic polynomials of degree $1$. Since $g$ is average zero and symmetric, it is orthogonal to $L^2(S^{d-1})$ to the eigenspaces that correspond to the first and second eigenvalues. Therefore, the inequality \eqref{eq: Poincare} holds for the third eigenvalue $2d$.
\end{proof}
Since $\grad_\sigma \sqrt{f}$ is symmetric and has  average $0$, the following holds:
\begin{cor} \label{cor: Poincare}
    \begin{equation}
    \int_{S^{d-1}} (\lap_\sigma \sqrt{f})^2 \dd \sigma \ge 2d \int_{S^{d-1}} |\grad_\sigma \sqrt{f}|^2 \dd \sigma=\frac{d}{2}\cdot  i(f).
\end{equation}
\end{cor}

From the equation \eqref{eq: H2norm}, Corollary \ref{cor: Poincare} is equivalent to
\begin{equation} \label{ineq: ineq2}
    \int_{S^{d-1}} f \left(A+\frac12 B \right)^2 \dd \sigma \ge 2d \cdot i(f).
\end{equation}
Recall that $S^{d-1}$ satisfies the Curvature-Dimension condition $CD(d-2,d-1)$ and this is equivalent to \eqref{eq: cd condtion}: 
\begin{equation}
    \norm{\grad_\sigma^2 g}^2 \ge \frac{1}{d-1} (\lap_\sigma g)^2.
\end{equation}
Therefore, for $g=f^{\beta+1}$, we have
\begin{equation} \label{eq: CD}
    \frac{1}{(\beta+1)^2}   \norm{\grad_\sigma^2 f^{\beta+1}}^2 \ge \frac{1}{(d-1)(\beta+1)^2} (\lap_\sigma f^{\beta+1})^2
\end{equation}
holds for every $\beta \in \R$.
Note that we recover the case $g=\log f$ by sending $\beta\to -1$, which is
\begin{equation} \label{ineq: logf}
    \norm{\grad_\sigma^2 \log f}^2 \ge \frac{1}{d-1} (\lap_\sigma \log f)^2.
\end{equation}
\begin{proof}[Proof of Theorem \ref{thm: Thm 1}]
Our goal is to optimize $\beta$ of \eqref{eq: CD} together with \eqref{ineq: ineq2} to estimate \eqref{eq: our goal}.
The following identities will be useful for our computation:
\begin{equation}
     \frac{\lap_\sigma f^{\beta+1}}{f^{\beta+1}} =\lap_\sigma \log f^{\beta+1} +|\grad_\sigma \log f^{\beta+1}|^2=(\beta+1) \lap_\sigma \log f +(\beta+1)^2 |\grad_\sigma \log f|^2,
\end{equation}
\begin{equation}
    \frac{\grad_\sigma f^{\beta+1}}{f^{\beta+1}}= \grad_\sigma \log f^{\beta+1}=(\beta+1) \grad_\sigma \log f.
\end{equation}
Multiply $f^{-2\beta-1}$ to \eqref{eq: CD} and integrate over $S^{d-1}$, then the right hand side of \eqref{eq: CD} becomes
\begin{equation}
    \frac{1}{d-1} \int_{S^{d-1}} f(A+(\beta+1)B)^2 \dd \sigma.
\end{equation}
By the Bochner's formula and integrating by parts, the left hand side of \eqref{eq: CD} becomes
\begin{align}
   &\frac{1}{(\beta+1)^2}  \int_{S^{d-1}} \left[ \frac12 f^{-2\beta-1}  \lap_\sigma (|\grad_\sigma f^{\beta+1}|^2)  -f^{-2\beta-1}\grad_\sigma \lap_\sigma f^{\beta+1} \cdot \grad_\sigma f^{\beta+1}-(d-2) f^{-2\beta-1} |\grad_\sigma f^{\beta+1}|^2 \right] \dd \sigma\\
    =&\frac{1}{(\beta+1)^2}  \int_{S^{d-1}} \left[ \frac12 \lap_\sigma f^{-2\beta-1}   |\grad_\sigma f^{\beta+1}|^2  +\frac{\beta+1}{\beta} \grad_\sigma \lap_\sigma f^{\beta+1} \cdot \grad_\sigma f^{-\beta}-(d-2) f^{-2\beta-1} |\grad_\sigma f^{\beta+1}|^2 \right] \dd \sigma\\
     =&  \int_{S^{d-1}} f \left[ \frac{-2\beta-1}{2} (\lap_\sigma \log f -(2\beta+1) |\grad_\sigma \log f|^2)  |\grad_\sigma \log f|^2 \right. \\
     &+ \left. (\lap_\sigma \log f +(\beta+1) |\grad_\sigma \log f|^2)(\lap_\sigma \log f -\beta |\grad_\sigma \log f|^2)-(d-2) f|\grad_\sigma \log f|^2 \right] \dd \sigma.
\end{align}
Using the definitions of $A$ and $B$, it simplifies to
\begin{equation}
    =\int_{S^{d-1}} f \left( A^2- \left(\beta-\frac12\right) AB+ \left(\beta^2+\beta+\frac12\right) B^2 \right) -(d-2)i(f).
\end{equation}
Hence, we obtain
\begin{equation} \label{eq: compute}
     \int_{S^{d-1}} f \left( A^2- \left(\beta-\frac12\right) AB+ \left(\beta^2+\beta+\frac12\right) B^2 \right) \dd \sigma -(d-2)i(f) \ge \frac{1}{d-1}\int_{S^{d-1}}f (A+(\beta+1) B)^2 \dd \sigma .
\end{equation}
Rearrange the terms of \eqref{eq: compute} and then multiply $\frac{d-1}{d-2}$ on both sides: 
\begin{equation} \label{ineq: ineq1}
    \int_{S^{d-1}} f \left( A^2-\left(\frac{d+1}{d-2}\beta-\frac{d-5}{2(d-2)} \right)AB+\left(\beta^2+\frac{d-3}{d-2}\beta+\frac{d-3}{2(d-2)} \right)B^2 \right) \dd \sigma \ge (d-1)i(f).
\end{equation}

Recall that our goal is to control the quantity \eqref{eq: our goal} with the inequalities \eqref{ineq: ineq2} and \eqref{ineq: ineq1}. We choose a parameter $ \theta_d(\beta) \in [0,1]$ so that
\begin{equation} \label{eq: linear combination}
    \theta_d(\beta)  \int_{S^{d-1}} f \left( A^2-\left(\frac{d+1}{d-2}\beta-\frac{d-5}{2(d-2)} \right)AB+\left(\beta^2+\frac{d-3}{d-2}\beta+\frac{d-3}{2(d-2)} \right)B^2 \right) \dd \sigma +(1-\theta_d(\beta)) \left(A+\frac12 B\right)^2
\end{equation}
is equal to
\begin{equation}
    (A+B)\left(A+\frac12 B\right) -\tau_d(\beta)\cdot B^2
\end{equation}
for some $\tau_d(\beta) \ge 0$. A straightforward computation shows that
\begin{equation}
    \theta_d(\beta)=-\frac{d-2}{d+1} \cdot \frac{1}{2\beta+1}, \quad \tau_d(\beta)=\frac{2(d-2)\beta+2d-3}{4(d+1)}. 
\end{equation}
In particular, if $d=2$, we have $\theta_2(\beta)=0$ and $\tau_2(\beta)=\frac{1}{12}$ for any $\beta$. 
For $d\ge 3$, we need $\beta$ to satisfy the condition
\begin{equation}  \label{eq: beta condition}
    -\frac{2d-3}{2d-4} \le  \beta \le -\frac12
\end{equation}
to guarantee $0\le \theta_d(\beta) \le 1$ and $0 \le \tau_d(\beta)$. 
Bootstrapping \eqref{ineq: ineq2} and \eqref{ineq: ineq1}, we obtain
\begin{equation}
   \int_{S^{d-1}} f \left( (A+B)\left(A+\frac12 B\right) -\tau_d(\beta)\cdot B^2 \right) \dd \sigma \ge \theta_d(\beta)  (d-1) i(f) + (1-\theta_d(\beta))  2d \cdot  i(f)= \left(2d+\frac{d-2}{2\beta+1} \right) i(f).
\end{equation}
It follows that
\begin{equation} \label{ineq: ineq3}
    \int_{S^{d-1}} f(A+B)\left(A+\frac12 B\right) \dd \sigma \ge  \left(2d+\frac{d-2}{2\beta+1} \right) i(f)
\end{equation}
for every $\beta$ satisfying \eqref{eq: beta condition}.
The right hand side of \eqref{ineq: ineq3} achieves its maximum value
\begin{equation} \label{eq: answer}
    \left(2d-\frac{(d-2)^2}{d-1} \right) i(f) = \left(d+3-\frac{1}{d-1} \right) i(f).
\end{equation}
when
\[
\beta=-\frac{2d-3}{2d-4}.
\]
Note that our choice of $\beta $ is $-\frac32 $ for $d=3$.

\end{proof}
\begin{remark}   
 For $d=2$, some of the computation above is invalid, but essentially the same computation works and \eqref{ineq: ineq3} remains true. The right hand side of \eqref{eq: answer} becomes $4i(f)$, thus $\Lambda_2 \ge 4$. 
\end{remark}
\begin{proof}[Proof sketch of Theorem \ref{thm: cd}]
Corollary \ref{cor: Poincare} becomes 
\begin{equation}
     \int_{M} f \left(A+\frac12 B \right)^2 \dd \sigma \ge \lambda_M  i(f)
\end{equation}
and the inequality \eqref{ineq: ineq1} becomes
\begin{equation}
     \int_{M} f \left( A^2-\left(\frac{n+2}{n-1}\beta-\frac{n-4}{2(n-1)} \right)AB+\left(\beta^2+\frac{n-2}{n-1}\beta+\frac{n-2}{2(n-1)} \right)B^2 \right) \dd \sigma \ge \frac{\rho n}{n-1}i(f).
\end{equation}
Similar to the proof of Theorem \ref{thm: Thm 1}, we choose
\begin{equation}
    \theta_M(\beta)=-\frac{n-1}{n+2}\cdot \frac{1}{2\beta+1}
\end{equation}
and 
\begin{equation}
    \tau_M(\beta)=\frac{2(n-1)\beta+2n-1}{4n}.
\end{equation}
We need $\beta$ to satisfy
\begin{equation}
    -\frac{2n-1}{2n-2} \le \beta \le -\frac12. 
\end{equation}
Therefore,
\begin{align}
   \int_M f (A+B)\left(A+\frac12 B \right) \dd \sigma 
    &\ge \theta_M(\beta)\cdot \frac{\rho n}{n-1} i(f)+(1-\theta_M(\beta))\lambda_M i(f)\\
    &=\left(\lambda_M + \frac{n-1}{n+2}\cdot \frac{1}{2\beta+1} \left( \lambda_M-\frac{\rho n }{n-1} \right) \right) i(f)\\
     &=\left (\lambda_M \cdot \frac{4n-1}{n(n+2)} + \frac{\rho n}{n-1} \cdot\frac{(n-1)^2}{n(n+2)} \right) i(f).
\end{align}
for $\beta=-\frac{2n-1}{2n-2}$.
\end{proof}
 
\begin{remark}
    As we mentioned earlier, $\beta=-1$ corresponds to the case $g=\log f$. This is how Guillen and Silvestre prove $\Lambda_3 \ge \frac{19}{4}$ in \cite{guillen2023global}. For instance, they are essentially choosing $\theta_3(-1)=\frac14$ and $\tau_3(-1)=\frac{1}{12}$.
    However, the proof of this section yields $\Lambda_3 \ge 5$ for $\beta=-1$, which is an improvement of their result. We briefly explain why this is the case. In this remark, we adopt the definitions in \cite{guillen2023global}, e.g., vector fields $b_1,b_2,b_3$.
     
         Suppose $A(\sigma)$ be a matrix given by $A_{ij}(\sigma)=b_i\cdot \grad( b_j\cdot \grad \log f)$ for $\sigma \in S^2$. Then, Lemma 9.9 of \cite{guillen2023global} can be improved to
        \begin{equation}\label{lem: improvement}
            \norm{M(\sigma)}^2 \ge \frac{1}{2} \tr(M(\sigma))^2 +\norm{N(\sigma)}^2 
        \end{equation}
        where $M(\sigma)$ and $N(\sigma)$ are the symmetric and anti-symmetric parts of $A(\sigma)$, respectively.
        An additional term $\norm{N}^2$ is added in comparison to the original Lemma 9.9 of \cite{guillen2023global}. We stress that \eqref{lem: improvement} does not hold for a generic matrix with rank at most $2$. In other words, the matrix $A(\sigma)$ is subject to more restrictive conditions than simply not being of full rank.
        If we rewrite \eqref{lem: improvement} using intrinsic geometric objects, then we get
        \begin{equation}
             \norm{\grad_\sigma^2 \log f }^2+ \frac{1}{2}|\grad_\sigma \log f|^2 \ge \frac{1}{2} (\lap_\sigma \log f)^2+ \frac{1}{2} |\grad_\sigma \log f |^2,
        \end{equation}
        which is exactly \eqref{ineq: logf}.
        We also provide an alternative proof of \eqref{lem: improvement} that does not rely on the geometric notions. It can be easily verified \eqref{lem: improvement} is equivalent to
       \begin{equation} \label{eq: trace}
           \tr(A(\sigma)^2)\ge \frac12 \tr(A(\sigma))^2.
       \end{equation}
       Without loss of generality, we can assume $\sigma=e_1$. Then,
       \begin{equation}
           b_1(\sigma)=0, \quad b_2(\sigma)=-e_3, \quad b_3(\sigma)=e_2.
       \end{equation}
       It follows that the first column of $A(e_1)$ is a zero column and the $(1,1)-$minor of $A(e_1)$ is symmetric. The inequality \eqref{eq: trace} reduces to
       \begin{equation}
           A(e_1)_{22}^2+A(e_1)_{23}^2+A(e_1)_{33}^2 \ge \frac12 (A(e_1)_{22}+A(e_1)_{33})^2,
       \end{equation}
       which is trivial.

\end{remark}

\section{The upper bounds for $\Lambda_3$ and $\alpha_3$}
In this section, we prove the upper bounds for $\Lambda_3$ and $\alpha_3$ by plugging in the following class of functions 
\begin{equation} \label{h def}
h(x,y,z)=(z^2+t)^2
\end{equation}
defined on $S^2$ for $t>0$ to the inequalities \eqref{eq: 2nd logSobolev} and \eqref{eq: log Sobolev}. It is a natural choice of functions since $\sqrt{h(x,y,z)}=z^2+t$ satisfies the equality for Corollary \ref{cor: Poincare}. However, the equality for \eqref{eq: CD} is not satisfied for this class of functions.
Thus, it is possible that considering other classes of functions might give a better estimate for $\Lambda_3$ and $\alpha_3$. For instance, functions that satisfy the equality for Corollary \ref{cor: Poincare} other than $h$ or functions that satisfy the equality for \eqref{eq: CD} are potential candidates. We expect the same class of functions to provide an upper bound for $\Lambda_d $ and $\alpha_d$ for general $d\ge 4$, but we do not explore this further in this paper.
\begin{remark}
    The components of $\grad_\sigma \sqrt{h}$ are spherical harmonic polynomials of degree $2$, not $1$ since the spherical gradient on $S^2$ does not take the radial direction into account. This can be also observed by performing the computation of this section following the notions of Guillen and Silvestre in \cite{guillen2023global}.
\end{remark}

\begin{remark}
    Since the inequalities \eqref{eq: 2nd logSobolev} and \eqref{eq: log Sobolev} are scale invariant, we can plug in the function \[
    h(x,y,z)=\left(1+\frac{z^2}{t}\right)^2
    \]
    instead of $(z^2+t)^2$, and send $t \to \infty$. Then, we obtain $\lambda_3=6 \ge \Lambda_3$ and $\lambda_3=6 \ge \alpha_3$ from \eqref{eq: 2nd logSobolev} and \eqref{eq: log Sobolev}, respectively.
    The same logic can be applied for general dimension $d$ to observe $\lambda_d=2d \ge \Lambda_d$ and $\lambda_d =2d \ge \alpha_d$.
\end{remark}


Recall a spherical coordinate system for $\R^3$:
\begin{equation}
    x= r\cos \theta \sin \varphi, \quad y=r\sin \theta \sin \varphi, \quad z=r\cos \varphi
\end{equation}
for $r \in [0,+\infty), \varphi \in [0,\pi], \theta \in [0,2\pi]. $
The unit sphere $S^2$ corresponds to $r=1$.

\begin{prop} \label{prop: sph}
    With the definition \eqref{h def} of $h$, we have
    \begin{equation}
        -\lap_\sigma\sqrt{h}=6\left(z^2-\frac13 \right), \quad |\grad_\sigma \sqrt{h}|^2=4z^2(1-z^2).
    \end{equation}
\end{prop}
\begin{proof}
     Since $z^2-\frac{1}{3}$ is a spherical harmonic polynomial of degree $2$ and has average $0$, we have
\begin{equation} \label{formula: eigenfunction}
    -\lap_\sigma\left(z^2-\frac13\right)=-\lap_\sigma \sqrt{h}=6\left(z^2-\frac13 \right).
\end{equation}
On the other hand,
\begin{align}
    |\grad_\sigma \sqrt{h}|^2
    &=r^2 (|\grad_{\R^3} \sqrt{h}|^2-|\partial_r \sqrt{h}|^2)\\
    &=r^2 (4z^2-4r^2 \cos^4 \varphi)\\
    &=4z^2(r^2-z^2).
\end{align}
By evaluating at a point $(x,y,z)\in S^2$, where $r=1$, we reach the conclusion of the proposition.
\end{proof}

\begin{proof}[Proof of Theorem \ref{thm: upper}]
Now, we are ready to compute both sides of the \eqref{eq: 2nd logSobolev}.
First, we compute the Fisher information for $h$:
\begin{equation} \label{eq: Fisherh}
    i(h)=4\int_{S^2} -\lap_\sigma \sqrt{h}\cdot  \sqrt{h} \dd \sigma = 24\int_{S^2} \left(z^2-\frac13\right)^2 \dd \sigma=\frac{32}{15}.
\end{equation}
On the other hand,
\begin{align} 
    \int_{S^2}h \Gamma_2(\log h, \log h ) \dd \sigma
    &=\int_{S^2}  h \left( \lap_\sigma \log h+|\grad_\sigma \log h|^2 \right) \left( \lap_\sigma \log h+ \frac{1}{2} |\grad_\sigma \log h|^2\right) \dd \sigma\\
    &= \int_{S^2}  4h \left( \lap_\sigma \log \sqrt{h}+2|\grad_\sigma \log \sqrt{h}|^2 \right) \left( \lap_\sigma \log \sqrt{h}+  |\grad_\sigma \log \sqrt{h}|^2\right) \dd \sigma
\end{align}
is equal to 
\begin{align}
    \int_{S^2} 4h \left(\frac{\lap_\sigma{\sqrt{h}}}{\sqrt{h}}+\frac{|\grad_\sigma \sqrt{h}|^2}{h} \right) \frac{\lap_\sigma{\sqrt{h}}}{\sqrt{h}}\dd \sigma
    &=4\int_{S^2} \left(\lap_\sigma{\sqrt{h}}+\frac{|\grad_\sigma \sqrt{h}|^2}{\sqrt{h}} \right) \lap_\sigma{\sqrt{h}} \dd \sigma\\
    &=24\int_{S^2} \left(z^2-\frac13\right) \left[6\left(z^2-\frac13\right)- \frac{4z^2(1-z^2)}{z^2+t} \right] \dd \sigma.
\end{align}
From \eqref{eq: 2nd logSobolev}, we get the estimate
\begin{equation} 
    (6-\Lambda_3) \int_{S^2} \left(z^2-\frac13\right)^2 \dd \sigma 
    \ge -\frac43\int_{S^2} \frac{z^2(1-z^2)(1-3z^2)}{z^2+t} \dd \sigma,
\end{equation}
which is equivalent to
\begin{equation} \label{eq: example LHS}
    \Lambda_3 \le 6+15 \int_{S^2} \frac{z^2(1-z^2)(1-3z^2)}{z^2+t} \dd \sigma.
\end{equation}
Since
\begin{equation}
    z^2(1-z^2)(1-3z^2)=(z^2+t)(3z^4-(3t+4)z^2+(t+1)(3t+1))-t(t+1)(3t+1),
\end{equation}
the right hand side of \eqref{eq: example LHS} equals to
\begin{align}
    &6+9-5(3t+4)+15(t+1)(3t+1)-15 t(t+1)(3t+1)\int_{S^2}\frac{1}{z^2+t} \dd \sigma\\
    &=5 (3t+1)(3t+2) - 15(t+1)(3t+1)\sqrt{t}\arctan\left(\frac{1}{\sqrt{t}}\right).
\end{align}
Note that we used an elementary integration
\begin{equation}
    \int_{S^2}\frac{1}{z^2+t} \dd \sigma =\frac{1}{4\pi} \int_{-1}^1 \frac{2\pi}{z^2+t} \dd z=\int_0^1 \frac{1}{z^2+t} \dd z=\left[\frac{1}{\sqrt{t}}\arctan\left(\frac{z}{\sqrt{t}}\right)\right]_0^1=\frac{1}{\sqrt{t}}\arctan\left(\frac{1}{\sqrt{t}}\right).
\end{equation}
Hence, we conclude that
\begin{equation} \label{eq: asymptotic}
    \Lambda_3 \le \min_{t>0} \left[ 5 (3t+1)(3t+2) - 15(t+1)(3t+1)\sqrt{t}\arctan\left(\frac{1}{\sqrt{t}}\right) \right] \approx 5.73892.
\end{equation}
The minimum value is approximately $5.73892$ for $t\approx 0.69214$ and computed numerically.
\end{proof}

Plugging in the same class of functions $h(x,y,z)=(z^2+t)^2$ to the logarithmic Sobolev inequality \eqref{eq: log Sobolev} provides the upper bound for $\alpha_3$.
\begin{proof}[Proof of Theorem \ref{thm: upperalpha}]
    
We have
\begin{equation}
    \int_{S^2} h \dd \sigma=\int_{S^2}(z^2+t)^2 \dd \sigma =\left(t^2+\frac{2t}{3}+\frac15 \right)
\end{equation}
and 
\begin{equation} \label{eq: logh}
    \int_{S^2} h \log h \dd \sigma =2\int_{S^2} (z^2+t)^2 \log (z^2+t)\dd \sigma=2 \int_0^1 (z^2+t)^2\log(z^2+t)\dd z.
\end{equation}
The right hand side of \eqref{eq: logh} is equal to 
\begin{equation}
    \frac{32}{15} t^2 \sqrt{t}\arctan\left(\frac{1}{\sqrt{t}}\right) +\left(t^2+\frac{2t}{3}+\frac15 \right)\log(t^2+2t+1)-\frac{4(120t^2+35t+9)}{225}
\end{equation}
through elementary computations, which we omit the detail. The logarithmic Sobolev inequality \eqref{eq: log Sobolev} becomes
\begin{equation} 
    \frac{32}{15} \ge 2\alpha_3 \left[    \frac{32}{15} t^2 \sqrt{t}\arctan\left(\frac{1}{\sqrt{t}}\right) +\left(t^2+\frac{2t}{3}+\frac15 \right)\log\left(\frac{t^2+2t+1}{t^2+\frac{2t}{3}+\frac15}\right)-\frac{4(120t^2+35t+9)}{225} \right].
\end{equation}
Therefore,
\begin{equation}
    \alpha_3 \le \min_{t>0}\left[ \left(  2 t^2 \sqrt{t}\arctan\left(\frac{1}{\sqrt{t}}\right) +\frac{15}{16}\left(t^2+\frac{2t}{3}+\frac15 \right)\log\left(\frac{t^2+2t+1}{t^2+\frac{2t}{3}+\frac15}\right)-\frac{120t^2+35t+9}{60} \right)^{-1} \right].
\end{equation}
The minimum value is approximately $5.8358$ for $t\approx 0.757585$ and computed numerically.

\end{proof}

\section*{Acknowledgments}
I would like to thank my advisor, Luis Silvestre, for helpful discussions and feedback. 

\bibliographystyle{abbrv}
\bibliography{citation}

\end{document}